\documentclass[11pt]{amsart}
\usepackage{amsmath}
\usepackage{hyperref}

\allowdisplaybreaks
\numberwithin{equation}{section}

\theoremstyle{plain} \newtheorem{theorem}{Theorem}[section]
\theoremstyle{plain} 
\theoremstyle{plain} 
\theoremstyle{plain} \newtheorem{proposition}[theorem]{Proposition}
\theoremstyle{plain} \newtheorem{question}[theorem]{Question}
\theoremstyle{remark} \newtheorem{remark}[theorem]{Remark}
\theoremstyle{definition} 
\theoremstyle{definition} 
\theoremstyle{remark} 

\newcommand{ \R}{ \mathbb R }

\begin{document}

\title{When do skew-products exist?}

\author[S.N. Evans]{Steven N. Evans}
\thanks{S.N.E. was supported in part by NSF grant DMS-0907639 and NIH grant 1R01GM109454-01}
\address{Department of Statistics\\
 367 Evans Hall  \#3860\\
 University of California \\
  Berkeley, CA  94720-3860 \\
   USA}
\email{evans@stat.berkeley.edu}

\author[A. Hening]{Alexandru Hening }
\thanks{A.H. was supported by EPSRC grant EP/K034316/1}
\address{Department of Statistics \\
 University of Oxford \\
 1 South Parks Road \\
 Oxford OX1 3TG \\
 United Kingdom}
 \email{hening@stats.ox.ac.uk}

\author{Eric Wayman}

\address{Department of Mathematics\\
         University  of California\\
         970 Evans Hall \#3840\\
         Berkeley, CA 94720-3840\\
         U.S.A.}

\email{ewayman@math.berkeley.edu}

\date{\today}

\begin{abstract}
The classical skew-product decomposition of planar Brownian motion
represents the process in polar coordinates as an autonomously
Markovian radial part and an angular part that is an independent
Brownian motion on the unit circle time-changed according to
the radial part.  Theorem~4 of \cite{L09} gives a broad generalization
of this fact to a setting where there is a diffusion on a manifold $X$
with a distribution that is equivariant under the smooth action of
a Lie group $K$.  Under appropriate conditions, there is a decomposition
into an autonomously Markovian ``radial'' part that lives on the space
of orbits of $K$ and an ``angular'' part that is an independent
Brownian motion on the homogeneous space $K/M$, where $M$ is the isotropy
subgroup of a point of $x$, that is time-changed with a time-change that
is adapted to the filtration of the radial part.
We present two apparent counterexamples to \cite[Theorem~4]{L09}.  In the
first counterexample the angular part is not a time-change of any Brownian motion
on $K/M$, whereas in the second counterexample the angular part is the time-change
of a Brownian motion on $K/M$ but this Brownian motion is not independent of the radial part.
In both of these examples $K/M$ has dimension $1$.  The statement and proof of \cite[Theorem~4]{L09}
remain valid when $K/M$ has dimension greater than $1$.  Our examples raise the question
of what conditions lead to the usual sort of skew-product decomposition when $K/M$ has dimension $1$
and what conditions lead to there being no decomposition at all or one in which the angular part is a time-changed
Brownian motion but this Brownian motion is not independent of the radial part.
\end{abstract}

\maketitle

\section{Introduction}

The archetypal skew-product decomposition of a Markov process is the decomposition of a Brownian motion in the plane $(B_t)_{t \geq 0}$ into its radial and angular part
\begin{equation}\label{e:BM_skew}
B_t = |B_t|  \exp(i \theta_t).
\end{equation}
Here the radial part $(|B_t|)_{t \geq 0}$ is a two-dimensional Bessel process  and $\theta_t = y_{\tau_t}$, where $(y_t)_{t \geq 0}$ is a one-dimensional Brownian motion that is independent of the radial part $(|B_t|)_{t \geq 0}$ and $\tau$ is a time-change that is adapted to the filtration generated by the process  $|B|$. Specifically, $\tau_t = \int_0^t \frac{1}{|B_s|^2}ds$. See Corollary 18.7 from \cite{K01} for more details.

The most obvious generalization of this result is obtained in \cite{G63}. The
process considered  is any time-homogeneous diffusion $(x_t)_{t \geq 0}$ with state space $\mathbb{R}^3$ that satisfies the additional assumptions that almost surely every path does not pass through the origin at positive times and that $(x_t)_{t \geq 0}$ is isotropic in the sense that the law of $(x_t)_{t \geq 0}$ is equivariant under the group of
orthogonal transformations;  that is, if we consider a point $(r, \theta) \in \mathbb{R}^3$ in spherical coordinates, where $r \in \mathbb{R}_+$ is the radial coordinate and $\theta$ is a point on the unit sphere $S^2$, and if we take $k \in O(3)$, the orthogonal group on $\mathbb{R}^3$, then
\begin{equation*}
P_{(r, k \theta)}\left( kA \right) = P_{(r, \theta)} \left( A \right)
\end{equation*}
for any Borel set $A$ in path space $C(\mathbb{R}_+, \mathbb{R}^3)$.  Here $P_{x}(A)$ is the probability a path started at $x$ belongs to the Borel set $A$ \cite[(2.2)]{G63}.
Theorem 1.2 of \cite{G63} states that we can decompose $(x_t)_{t \geq 0}$ as $x_{t} = r_{t} \theta_{t}$ where the radial motion $(r_t)_{t \geq 0}$ is a time-homogeneous Markov process on $\mathbb{R}_+$ and the angular process $(\theta_{t})_{t \geq 0}$  can be written as $\theta_{t} = B_{\tau_{t}}$, with $(B_t)_{t \geq 0}$  a spherical Brownian motion independent of the radial part and with the time-change $(\tau_{t})_{t \geq 0}$ adapted to the filtration generated by the radial part.

More generally, one can consider a group $G$ acting on $\mathbb{R}^n$ and $(x_t)_{t \geq 0}$ a Markov process on $\mathbb{R}^n$ such that the distribution of $(x_t)_{t\geq 0}$ satisfies the equivariance condition
\[
P_{gx}(gA) = P_x(A)
\]
for any Borel set $A$ in path space. The existence of a skew-product decomposition for this setting
is explored in \cite{Chy08} when $(x_t)_{t \geq 0}$ is a Dunkl process and $G$ is the group of distance preserving transformations of $\mathbb{R}^n$.

The paper \cite{PR88} investigates the skew-product decomposition of a Brownian motion on a $C^{\infty}$ Riemannian manifold $(M,g)$ which can be written as a product of a radial manifold $R$ and an angular manifold $\Theta$, both of which are assumed to be smooth and connected. Provided the Riemannian metric respects the product structure of the manifold in a suitable manner, \cite[Theorem~4]{PR88} establishes the existence of a skew-product decomposition such that the radial motion is a Brownian motion with drift on $R$ and the angular motion is a time-change of a
Brownian motion on $\Theta$ that is independent of the radial motion.

A broadly applicable skew-product decomposition result is obtained in \cite{L09} for a general continuous Markov process $(x_t)_{t \geq 0}$ with state space  a smooth manifold $X$ and distribution that is equivariant under the smooth action of a Lie group $K$ on $X$.  Here the decomposition of $(x_t)_{t \geq 0}$ is into a radial part $(y_t)_{t \geq 0}$ that is a Markov process on the submanifold $Y$ which is transversal to the orbits of $K$ and an angular part $(z_t)_{t \geq 0}$ that is a process on a general $K$-orbit which can be identified with the homogeneous space $K/M$, where $M$ is the isotropy subgroup of $K$ that is assumed to be the same for all elements $x \in X$.  Theorem 4 of \cite{L09} asserts that under suitable conditions the process $(x_t)_{t \geq 0}$ has the same distribution as $(B(a_t)y_t)_{t \geq 0}$, where the radial part $(y_t)_{t \geq 0}$ is a diffusion on $Y$, $(B_t)_{t \geq 0}$ is a Brownian motion on $K/M$ that is independent of $(x_t)_{t \geq 0}$, and $(a_t)_{t \geq 0}$ a
time-change that is adapted to the filtration generated by $(y_t)_{t \geq 0}$.

The present paper was motivated by our desire to understand better the structural features that give rise to skew-product decompositions of diffusions that are equivariant under the action of a group and what it is about the absence of these features which cause such a decomposition not to hold.  In attempting to do so, we read the paper \cite{L09}.  We found an apparent counterexample to the main result, Theorem 4 of that paper in which there is a decomposition of
the process into an autonomously Markov radial process on $Y$ and an angular part that is a Brownian motion on $K/M$ time-changed according to the radial process, but this Brownian motion is {\bf not}, contrary to the claim of \cite{L09}, independent of the radial process, see Section~\ref{s_counter} for an exposition of the counterexample.  This seeming contradiction appears because the assumption from \cite{L09} that $K/M$ is irreducible is not strong enough to ensure the nonexistence of a nonzero $M$-invariant tangent vector in the special case when, as in our construction, $K/M$ has dimension $1$.
It is
the nonexistence of such a tangent vector that is used in the
proof in \cite{L09} to deduce the independence of the radial process and the Brownian motion.  Professor Liao pointed out to us that \cite[Theorem~4]{L09} holds under the conditions in \cite{L09}
for $\text{dim}(K/M)> 1$
and that result also holds when $K/M$ has dimension
$1$ if we further assume that there is no $M$-invariant tangent vector.

An anonymous referee pointed out an even simpler counterexample to \cite[Theorem~4]{L09} which we present in Section~\ref{s:rotated}. Namely, one takes
\[
x_t= \Theta_t \begin{pmatrix}U_t\\V_t\end{pmatrix}
\]
where $\begin{pmatrix}U_t\\V_t\end{pmatrix}$ is a planar Brownian motion and $\Theta_t\in SO(2)$ is the matrix that represents rotation about the origin through an angle $t$. We show that in this case that there is no skew-product decomposition for a somewhat different (and perhaps less interesting) reason: the angular part of $(x_t)_{t\geq 0}$ cannot be written as a time-changed Brownian motion on the unit circle in the plane.
The apparent contradiction to \cite[Theorem~4]{L09}is again due to the irreducibility of  $K/M$ being inadequate to ensure the non-existence of an $M$ invariant tangent vector when $K/M$ has dimension $1$.

We present both of these counterexamples here because
they illustrate two rather different ways in which
things can go wrong.  The latter counterexample
shows that under what look like reasonable conditions one might fail to have a skew-product decomposition because the angular part can't be time-changed to be Brownian, whereas the former counterexample does involve an angular part that is a time-changed Brownian motion, but it is just that this Brownian motion isn't independent of the radial process.  We hope that by presenting these two examples we will
prompt further investigation into what general conditions lead to the subtle failure of the usual skew-product decomposition in the first counterexample and what ones lead to the grosser failure in the second counterexample.

The outline of the remainder of the paper is the following.

In Section \ref{s:BM} we check that the classical skew-product decomposition of planar Brownian motion fits in the setting from \cite{L09}, even though the proof of \cite[Theorem~4]{L09} does not, as we have noted, apply to ensure the existence of the skew-product decomposition when, as here, the dimension
of $K/M$ is $1$.

In Section~\ref{s:rotated} we describe the counterexample mentioned above of a planar Brownian motion that is rotated at a constant rate for which
the angular part is not a time-changed Brownian motion on the unit circle in the plane.

In Section \ref{s_counter} we construct the counterexample of a diffusion for which the angular part is a time-changed Brownian motion on the appropriate homogeneous space, but this Brownian motion is not independent of the radial part.  Here the diffusion $(x_t)_{t \geq 0}$ has state space the manifold of $2 \times 2$ matrices that have a positive determinant.  This diffusion can be represented via the well-known QR decomposition as the product of an
autonomously Markov
``radial'' process $(T_t)_{t\geq 0}$ on the manifold of $2 \times 2$ upper-triangular matrices with positive diagonal entries
and a time-changed ``angular'' process $(U_{R_t})_{t \geq 0}$, where $(U_t)_{t \ge 0}$ is a Brownian motion on the group $SO(2)$ of $2 \times 2$ orthogonal matrices with determinant one and the time-change $(R_t)_{t \ge 0}$ is adapted to the
filtration of the radial process.  However,  the processes $(U_t)_{t \geq 0}$ and $(T_t)_{t \geq 0}$ are {\bf not} independent.

We end this introduction by noting that analogous  skew-product decompositions of superprocesses  have been studied in \cite{P91, EM91, H00}. The continuous Dawson-Watanabe (DW) superprocess is a rescaling limit of a system of branching Markov processes while the Fleming-Viot (FV) superprocess is a rescaling limit of the empirical distribution of a system of particles undergoing Markovian motion and multinomial resampling. It
is shown in \cite{EM91} that a FV process is a DW process conditioned to have total mass one. More generally, it is demonstrated in \cite{P91} that the distribution of the DW process conditioned on the path of its total mass process is equal to the distribution of a time-change of a FV process that has a suitable underlying time-inhomogeneous Markov motion. The latter result is extended
to measure-valued processes that may have jumps in \cite{H00}.

A sampling of other results involving skew-products can be found
in  \cite{Tay92,La09,El10, Ba06}.

\section{Example 1: Planar Brownian motion}\label{s:BM}

Let $(x_t)_{t \ge 0}$ be a planar Brownian motion.

Following the notation of \cite{L09}, we consider the following set-up.
\begin{enumerate}
\item Let $X=\R^2 \setminus \{(0,0)^T\}$.
\item Let $K$ be the Lie group $SO(2)$ of $2\times2$ orthogonal matrices
 with determinant $1$.  This group acts on $X$ by $A \mapsto Q^{-1} A$ for
 $Q \in K$ and $A \in X$.
\item The quotient of $X$ with respect to the action of $K$ can be identified with the positive $x$ axis. Note that the orbits of $K$ are just circles centered at the origin.
\item The isotropy subgroup of $K$ for an element $x \in X$ is, as usual, the
subgroup $\{k \in K : kx = x\}$.  Since every element of $X$ is an invertible
matrix, this subgroup is always the trivial group
consisting of just the identity.
In particular, this subgroup is the same for every
$y$ in the interior of $Y$, as required in \cite[pg~168]{L09}.
We denote this subgroup by $M$.
\end{enumerate}

It is straightforward to check that  $(x_t)_{t \ge 0}$ satisfies
all the assumptions of \cite[Theorem~4]{L09}. We refer the reader to Sections \ref{s:rotated} and \ref{s_counter} for details of how to verify these assumptions in more complicated examples.
\begin{remark}\label{r:BM}
In this example, $\text{dim}(K/M)=1$ and there is the skew-product decomposition \eqref{e:BM_skew}.
\end{remark}

\section{Example 2: Rotated planar Brownian motion}\label{s:rotated}
Write $((U_t,V_t)^T)_{t \ge 0}$
for a planar Brownian started from $(x,y)^T$ (where $T$ denotes transpose, so we are
thinking of column vectors).  The process $(x_t)_{t\geq 0}:=\left((x_t^1,x_t^2)^T\right)_{t\geq 0}$ started from
$(x,y)^T$ is defined by
\begin{equation}\label{e:rotatedBM}
\begin{pmatrix}x_t^1\\x_t^2\end{pmatrix} = \Theta_t \begin{pmatrix}U_t\\V_t\end{pmatrix},
\end{equation}
where $\Theta_t$ is the matrix that represents rotating though an angle $t$.  Thus,
\begin{equation}
\begin{split}
x^1_t &= \cos(t) U_t - \sin(t) V_t\\
x^2_t &= \sin(t) U_t + \cos(t) V_t.
\end{split}
\end{equation}
Then,
\begin{equation*}
\begin{split}
dx^1_t &= \cos(t) dU_t - U_t \sin(t) dt - \sin(t) dV_t - V_t \cos(t) dt\\
dx^2_t &= \sin(t) dU_t + U_t \cos(t) dt + \cos(t) dV_t - V_t \sin(t) dt,
\end{split}
\end{equation*}
which becomes
\begin{equation*}
\begin{split}
dx^1_t &= \cos(t) dU_t - \sin(t) dV_t - Y_t dt\\
dx^2_t &= \sin(t) dU_t + \cos(t) dV_t + X_t dt.
\end{split}
\end{equation*}
If we define martingales $(B_t)_{t \ge 0}$
and $(C_t)_{t \ge 0}$ by
\[
dB_t = \cos(t) dU_t - \sin(t) dV_t
\]
and
\[
dC_t =  \sin(t) dU_t + \cos(t) dV_t,
\]
then $[B]_t = t$, $[C]_t = t$ and $[B,C]_t = 0$, so
the process $((B_t,C_t)^T)_{t \ge 0}$ is a
planar Brownian motion and the process $\left((x_t^1,x_t^2)^T\right)_{t\geq 0}$ satisfies
the SDE
\begin{equation}\label{e:SDE_rotatedBM}
\begin{split}
dx^1_t &= dB_t - Y_t dt\\
dx^2_t &= dC_t + X_t dt.
\end{split}
\end{equation}
Following the notation of \cite{L09}, we consider the following set-up.
\begin{enumerate}
\item Let $X=\R^2 \setminus \{(0,0)^T\}$.
\item Let $K$ be the Lie group $SO(2)$ of $2\times2$ orthogonal matrices
 with determinant $1$.  This group acts on $X$ by $A \mapsto Q^{-1} A$ for
 $Q \in K$ and $A \in X$.
\item The quotient of $X$ with respect to the action of $K$ can be identified with the positive $x$ axis. Note that the orbits of $K$ are just circles centered at the origin.
\item The isotropy subgroup of $K$ for an element $x \in X$ is, as usual, the
subgroup $\{k \in K : kx = x\}$.  Since every element of $X$ is an invertible
matrix, this subgroup is always the trivial group
consisting of just the identity.
In particular, this subgroup is the same for every
$y$ in the interior of $Y$, as required in \cite[pg~168]{L09}.
We denote this subgroup by $M$.
\item Let $(x_t)_{t \ge 0}$ be the $X$-valued process that is defined in \eqref{e:rotatedBM}.
\end{enumerate}

We now check that  $(x_t)_{t \ge 0}$ satisfies
all the assumptions of \cite[Theorem~4]{L09}.
These are as follows:
\begin{enumerate}
\item The process $(x_t)_{t \ge 0}$ is a Feller process with continuous
sample paths.
\item The distribution of $(x_t)_{t \ge 0}$ is equivariant under
the action of $K$.  That is, for $k \in K$ the distribution of  $(k x_t)_{t \ge 0}$
when $x_0 = x_*$  is the same as the distribution of
$(x_t)_{t \ge 0}$ when $x_0 = k x_*$ \cite[(2)]{L09}.

\item The set $Y$ is a submanifold of $X$ that is transversal to the action of $K$
\cite[(3)]{L09}.

\item For any $y \in Y^0$ (that is, the relative interior of $Y$ -- which
in this case is just $Y$ itself) $T_y X$,
the tangent space of $X$ at $y$, is the direct sum of tangent spaces
$T_y(Ky) \bigoplus T_y Y$  \cite[(5)]{L09}.
\item The homogeneous space $K/M$ is irreducible; that is, the action of $M$
on $T_o(K/M)$ (the tangent space at the coset $o$ containing the identity)
has no nontrivial invariant subspace  \cite[pg~177]{L09}.
\end{enumerate}
\vspace{5mm}

These assumptions are verified as follows:

\begin{enumerate}
\item This follows from the representation \eqref{e:SDE_rotatedBM}.

\item Since $\Theta_t\in SO(2)$ we have by \eqref{e:rotatedBM} that for any $Q\in SO(2)$
\[
Q x_t = Q \Theta_t\begin{pmatrix}U_t\\V_t\end{pmatrix}.
\]
Since $ Q \Theta_t \in SO(2)$ the condition holds because planar Brownian motion is equivariant under the action of $SO(2)$.

\item This is immediate.

\item $T_y(Ky) = \text{Span}\left\{ (0,1)^T\right\}$ and $T_y(Y) = \text{Span}\left\{(1,0)^T\right\}$ so that
\[
\R^2 = T_y X = T_y(Ky) \oplus T_y(Y)
\]

\item The tangent space of $Ky$ is one-dimensional so $K/M$ is irreducible.
\end{enumerate}
Consequently,  $(x_t)_{t \ge 0}$ satisfies all the hypotheses of \cite[Theorem~4]{L09}.

Write $(R_t)_{t\geq 0}$ for the radial process
\[
R_t := |(x_t^1, x_t^2)^T| = |(U_t, V_t)^T|,
\]
and let $(L_t)_{t\geq 0}$ be the angular part of
$((U_t, V_t)^T)_{t \ge 0}$.  We can think of $(L_t)_{t\geq 0}$ as living on the unit circle in the complex plane.  In
polar coordinates, we have
\[
x_t = (R_t, L_t \exp(it)).
\]
By the usual skew-product for planar Brownian motion recalled in \eqref{e:BM_skew} we have that $L_t =
\exp(i W_{T_t})$, where $W$ is a standard Brownian motion on the line
independent of $R$ and $T$ is a time-change defined
from $R$.  Therefore
\[
x_t = (R_t, \exp(i (W_{T_t}+t))).
\]

\begin{proposition}\label{p_mainprop_rotatedBM}
The process $(x_t)_{t\geq 0}$ cannot be written as
\[
x_t = (R_t, \exp(i Z_{S_t})),
\]
where $Z$ is a Brownian motion (possibly with drift) on the line
independent of $R $ and $S$ is an increasing process adapted to the filtration generated by $R$.
\end{proposition}

\begin{proof}
If such a representation was possible, then we would have $Z_t = \tilde Z_t + a t$ for some constant $a \in \R$, where $\tilde Z_t$ is a standard Brownian motion. This would imply that
\begin{equation*}
\begin{split}
\tilde Z &= W\\
S &= T\\
\exp(i a S_t) &= \exp(i t).
\end{split}
\end{equation*}
However, this is not possible:  it would mean that
\[
\exp(i t) = \exp(i a T_t),
\]
but $T_t$ is certainly not a constant multiple of $t$
for all $t \ge 0$.
\end{proof}

\begin{remark}\label{r:rotated}
In this example $K/M$ is the unit circle, which has dimension $1$, and there is no skew-product decomposition. The angular part cannot be written as the time-change of any Brownian motion on the unit circle.
\end{remark}

\section{Example 3: A matrix valued process}\label{s_counter}

Recall the well-known QR decomposition which says that any square matrix
can be written as the product of an orthogonal matrix  and an
upper triangular matrix, and that this decomposition is unique for invertible matrices if we require  the diagonal entries in the
upper triangular matrix to be positive (see, for example, \cite{horn}).
This decomposition is essentially a special case of the Iwasawa decomposition
for semisimple Lie groups.

In the $2 \times 2$ case, uniqueness also holds for QR decomposition of invertible matrices if we require the orthogonal matrix to have determinant one and there are simple
explicit formulae for the factors.  Indeed,
if
\begin{equation}\label{eq_qrfact}
A = \left(\begin{matrix}a&b\\ c&d\end{matrix}\right)
\end{equation}
and $\det A = ad - bc \ne 0$,
then $A = \tilde Q \tilde R$, where
\begin{equation}\label{e:Q1}
 \tilde{Q}=\frac{1}{\sqrt{a^2+c^2}} \left(\begin{matrix}a&-c\\ c&a\end{matrix}\right) \in SO(2)
\end{equation}
and
\begin{equation}\label{e:R1}
\tilde{R}=\left(\begin{matrix}\sqrt{a^2 + c^2}&\frac{ab + cd}{\sqrt{a^2 + c^2}}\\ 0& \frac{ad-bc}{\sqrt{a^2 + c^2}}\end{matrix}\right).
\end{equation}

In this setting, we consider a $2 \times 2$ matrix of independent Brownian motions and time-change it to produce a Markov process with the property that if the determinant is positive at time $0$, then it stays positive at all times.  This ensures that uniqueness of the $QR$-factorization holds at all times and also that the time-changed process falls into the setting of \cite{L09}.

Following the notation of \cite{L09}, we consider the following set-up.
\begin{enumerate}
\item Let $X$ be the manifold of $2\times2$ matrices over $\R$
with strictly positive determinant equipped with the topology it inherits
as an open subset of $\R^{2 \times 2} \cong \R^4$.
\item Let $K$ be the Lie group $SO(2)$ of $2\times2$ orthogonal matrices
 with determinant $1$.  This group acts on $X$ by $A \mapsto Q^{-1} A$ for
 $Q \in K$ and $A \in X$.
\item The quotient of $X$ with respect to the action of $K$ can,
via the QR decomposition, be identified with
the set $Y$ of upper triangular $2\times2$ matrices with
strictly positive diagonal entries.
\item The isotropy subgroup of $K$ for an element $x \in X$ is, as usual, the
subgroup $\{k \in K : kx = x\}$.  Since every element of $X$ is an invertible
matrix, this subgroup is always the trivial group
consisting of just the identity.
In particular, this subgroup is the same for every
$y$ in the interior of $Y$, as required in \cite[pg~168]{L09}.
We denote this subgroup by $M$.
\item Let $(x_t)_{t \ge 0}$ be the $X$-valued process that satisfies
 the stochastic differential equation (SDE)
\begin{equation}\label{e_SDE}
    dx_t =\left(\begin{matrix} dx^{1,1}_t&dx^{1,2}_t\\dx^{2,1}_t&dx^{2,2}_t\end{matrix} \right) = \left(\begin{matrix} f(x_t) \, dA^{1,1}_t& f(x_t) \, dA^{1,2}_t\\f(x_t) \, dA^{2,1}_t & f(x_t) \, d A^{2,2}_t \end{matrix} \right), \quad x_0 \in X,
\end{equation}
where $A^{1,1}_t$, $A^{1,2}_t$, $A^{2,1}_t$, and $ A^{2,2}_t$ are independent standard one-dimensional Brownian motions, and $f(x):= \frac{\det(x)}{\text{tr}(x'x)+1}$ with $\det$ and $\text{tr}$ denoting the determinant and the trace.
We establish below that \eqref{e_SDE} has a unique strong
solution and that this
solution does indeed take values in $X$.
\end{enumerate}

It follows from the QR decomposition
that $x_t = Q_t T_t$, where, in the terminology of \cite{L09},
the ``angular part'' $Q_t$ belongs to $K$ and the ``radial part''
$T_t$ belongs to $Y$.
We will show that $(T_t)_{t \ge 0}$ is an autonomous diffusion
on $Y$ and that $Q_t = U_{R_t}$, where $(U_t)_{t \ge 0}$
is a Brownian motion on $K$ and $(R_t)_{t \ge 0}$ is an increasing process adapted to
the filtration generated by $(T_t)_{t \ge 0}$.
However, we will establish that \textbf{it is not possible} to take the
Brownian motion $(U_t)_{t \ge 0}$ to be independent of
the process $(T_t)_{t \ge 0}$.
This will contradict the claim of \cite[Theorem~4]{L09} once we have
also checked that the conditions of that result hold.

Note that if we consider $f$ as a function on the space
$\R^{2 \times 2} \cong \R^4$ of all $2 \times 2$ matrices, then it
has bounded partial derivatives, and hence it is globally
Lipschitz continuous.  Consequently, if we allow the initial
condition in $\eqref{e_SDE}$ to be an arbitrary element of $\R^{2 \times 2}$,
then the resulting SDE has a unique strong solution (see, for example, \cite[Ch~5,~Thm~11.2]{RW00}).  Moreover, the resulting process
is a Feller process on $\mathbb{R}^{2 \times 2}$
(see, for example, \cite[Ch~5, Thm~22.5]{RW00}).

We now check that $(x_t)_{t \ge 0}$ actually takes values in $X$.
That is, we show that if $x_0$ has positive determinant, then $x_t$ also has positive determinant for all $t \ge 0$.  It follows from It\^o's Lemma that
\begin{equation*}
    [\det(x_\cdot)]_t = \int_0^t \mathrm{tr}(x_s'x_s)f^2(x_s)\,ds,
\end{equation*}
\begin{equation*}
    [\mathrm{tr}(x_\cdot'x_\cdot)]_t = \int_0^t4\mathrm{tr}(x_s'x_s)f^2(x_s)\,ds,
\end{equation*}
and
\begin{equation*}
     [\det(x_\cdot),\mathrm{tr}(x_\cdot'x_\cdot)] = \int_0^t4\det(x_s)f^2(x_s)\,ds.
\end{equation*}
Thus,  $((\det(x_t), \mathrm{tr}(x_t'x_t)))_{t \ge 0}$ is
a Markov process and there exist independent standard one-dimensional Brownian motions
$(B^1_t)_{t \ge 0}$ and $(B^2_t)_{t \ge 0}$ such that
\[
d \, \det(x_t) = \sqrt{\mathrm{tr}(x_t'x_t)} f(x_t) \, dB^1_t
\]
and
\[
\begin{split}
d \, \mathrm{tr}(x_t'x_t) & = \frac{4\det(x_t) f(x_t)}{\sqrt{\mathrm{tr}(x_t'x_t)}} \, dB^1_t + \sqrt{\frac{4 \mathrm{tr}^2(x_t'x_t) - 16\det(x_t)^2} {\mathrm{tr}(x_t'x_t)}}f(x_t) \, dB^2_t \\
& \quad + 4f^2(x_t) \, dt.
\end{split}
\]
When we substitute for $f$, the above equations transform into
\[
d \, \det(x_t) = \frac{\det(x_t) \sqrt{\mathrm{tr}(x_t'x_t)}}{\mathrm{tr}(x_t'x_t)+1} \, dB^1_t
\]
and
\[
\begin{split}
d \, \mathrm{tr}(x_t'x_t) &=
\frac{4(\det(x_t))^{2}}{\sqrt{\mathrm{tr}(x_t'x_t)} (\mathrm{tr}(x'x)+1)} \, dB^1_t + \sqrt{\frac{4 \mathrm{tr}^2(x_t'x_t) - 16\det(x_t)^2} {\mathrm{tr}(x_t'x_t)}}\frac{\det(x_t)}{\mathrm{tr}(x_t'x_t)+1} \, dB^2_t \\
& \quad + 4 \left( \frac{\det(x_t)}{\mathrm{tr}(x_t'x_t)+1} \right)^2 \, dt. \\
\end{split}
\]
In particular,  the process
$(\det(x_t))_{t \ge 0}$ is the stochastic exponential
of the local martingale $(M_t)_{t \ge 0}$, where
\[
M_t = \int^t_0 \frac{\sqrt{\mathrm{tr}(x_s'x_s)}}{\mathrm{tr}(x_s'x_s)+1} \, dB^1_s.
\]
Since $x_0 \in X$, we have $\det(x_0) > 0$, and hence
\[
\det(x_t) = \det (x_0) \exp\left(M_t-M_0-\frac{1}{2}[M]_t\right)
\]
 is strictly positive for all $t \ge 0$. This shows that
$(x_t)_{t \ge 0}$ takes values in $X$.

We now check that  $(x_t)_{t \ge 0}$ satisfies
all the assumptions of \cite[Theorem~4]{L09}.
These are as follows:
\begin{enumerate}
\item The process $(x_t)_{t \ge 0}$ is a Feller process with continuous
sample paths.
\item The distribution of $(x_t)_{t \ge 0}$ is equivariant under
the action of $K$.  That is, for $k \in K$ the distribution of  $(k x_t)_{t \ge 0}$
when $x_0 = x_*$  is the same as the distribution of
$(x_t)_{t \ge 0}$ when $x_0 = k x_*$ \cite[(2)]{L09}.

\item The set $Y$ is a submanifold of $X$ that is transversal to the action of $K$
\cite[(3)]{L09}.

\item For any $y \in Y^0$ (that is, the relative interior of $Y$ -- which
in this case is just $Y$ itself) $T_y X$,
the tangent space of $X$ at $y$, is the direct sum of tangent spaces
$T_y(Ky) \bigoplus T_y Y$  \cite[(5)]{L09}.
\item The homogeneous space $K/M$ is irreducible; that is, the action of $M$
on $T_o(K/M)$ (the tangent space at the coset $o$ containing the identity)
has no nontrivial invariant subspace  \cite[pg~177]{L09}.
\end{enumerate}
\vspace{5mm}

The verifications of (1)--(5) proceed as follows:
\begin{enumerate}
\item We have already observed that solutions of \eqref{e_SDE}
with initial conditions in $\R^{2 \times 2}$ form a Feller process
and that this process stays in the open set $X$ if it starts in $X$,
and so $(x_t)_{t \ge 0}$ is a Feller process on $X$.

\item  Suppose that $(x_t)_{t \ge 0}$ is a solution of \eqref{e_SDE}
with $x_0 = x_*$ and $(\hat x_t)_{t \ge 0}$ is a solution of \eqref{e_SDE}
with $\hat x_0 = k x_*$ for some $k \in K$.  We have to show that if
we set $\tilde x_t = k^{-1} \hat x_t$, then $(\tilde x_t)_{t \ge 0}$
has the same distribution as $(x_t)_{t \ge 0}$.  Note that
$\det \tilde x_t = \det \hat x_t$ and
$\tilde x'_t \tilde x_t = \hat x_t' \hat x_t$, so that
$f(\tilde x_t) = f(\hat x_t)$.  Thus,
\[
d \tilde x_t =
f(\tilde x_t)
k^{-1}
\begin{pmatrix}
dA^{1,1}_t & dA^{1,2}_t\\
dA^{2,1}_t & dA^{2,2}_t
\end{pmatrix}, \quad \tilde x_0=x_*.
\]
Now the columns of the matrix
\[
\begin{pmatrix} A^{1,1}_t & A^{1,2}_t \\ A^{2,1}_t& A^{2,2}_t \end{pmatrix}
\]
are independent standard two-dimensional Brownian motions, and so the
same is true of the columns of the matrix
\[
k^{-1}
\begin{pmatrix} A^{1,1}_t & A^{1,2}_t \\ A^{2,1}_t& A^{2,2}_t \end{pmatrix}
\]
by the equivariance of standard two-dimensional Brownian motion under the action of $SO(2)$.
Hence,
\[
k^{-1}
\begin{pmatrix} A^{1,1}_t & A^{1,2}_t \\ A^{2,1}_t & A^{2,2}_t \end{pmatrix}  =
\begin{pmatrix} \alpha^{1,1}_t & \alpha^{1,2}_t \\ \alpha^{2,1}_t & \alpha^{2,2}_t \end{pmatrix},
\]
where
$(\alpha^{1,1}_t)_{t \ge 0}$, $(\alpha^{1,2}_t)_{t \ge 0}$,
$(\alpha^{2,1}_t)_{t \ge 0}$, and $(\alpha^{2,2}_t)_{t \ge 0}$
are independent standard Brownian motions.
Since,
\begin{equation*}
d\tilde{x}_t =
f(\tilde{x}_t) \begin{pmatrix} d\alpha^{1,1}_t & d\alpha^{1,2}_t \\ d\alpha^{2,1}_t & d\alpha^{2,2}_t \end{pmatrix}, \quad \tilde{x_0} = x_0,
\end{equation*}
the existence and uniqueness of strong solutions to \eqref{e_SDE}
establishes that the distributions of $(x_t)_{t \ge 0}$ and
$(\tilde{x}_t)_{t \ge 0}$ are equal.

\item  It follows from the existence of the
$QR$ decomposition for invertible matrices that $X$
is the union of the orbits $Ky$ for $y \in Y$, and it follows
from the uniqueness of the decomposition for such matrices
that the orbit $Ky$ intersects $Y$ only at $y$.

\item  Since the tangent space of $K=SO(2)$ at the identity is the
vector space of $2 \times 2$ skew-symmetric matrices and the tangent space of
$Y$ at the identity is the vector space of $2 \times 2$ upper-triangular matrices,
we have to show that if $W$ is a fixed invertible
upper-triangular $2 \times 2$ matrix and $M$ is a fixed $2 \times 2$ matrix, then
\begin{equation*}
M = S W + V
\end{equation*}
for a unique skew-symmetric $2 \times 2$ matrix $S$ and unique upper-triangular $2 \times 2$ matrix $V$.
Let
\begin{equation*}
M:=\begin{pmatrix} m_{11}&m_{12}\\m_{21}&m_{22}\end{pmatrix}
\quad \text{and} \quad
W:=
\begin{pmatrix} w_{11}&w_{12}\\0&w_{22}\end{pmatrix}.
\end{equation*}
It is immediate that
\begin{equation*}
S =
\begin{pmatrix} 0 &-\frac{m_{21}}{w_{11}}\\\frac{m_{21}}{w_{11}}&0 \end{pmatrix}
\end{equation*}
and
\begin{equation*}
V = \begin{pmatrix} m_{11} &\frac{m_{12} w_{11} + m_{21} w_{22}}{w_{11}}\\0&\frac{m_{22} w_{11} - m_{21} w_{12}}{w_{11}}\end{pmatrix}.
\end{equation*}

\item We have already noted that the
 tangent space of $K$ at the identity is the vector space of
skew-symmetric $2\times2$ matrices.  This vector space
is one-dimensional and
so this condition holds trivially.

\end {enumerate}

We have now shown that  $(x_t)_{t \ge 0}$ satisfies all the hypotheses of \cite[Theorem~4]{L09}.  However, we have the following result.

\begin{proposition}\label{p_mainprop}
In the decomposition $x_t= Q_t T_t$ the
$Y$-valued process $(T_t)_{t \ge 0}$ is Markov and
the $K$-valued process $(Q_t)_{t \ge 0}$
may be written as
$Q_t = U_{R_t}$,
where $(U_t)_{t \ge 0}$ is a $K$-valued Brownian motion and
$(R_t)_{t \ge 0}$ is an increasing continuous process
such that $R_0 = 0$ and $R_t - R_s$ is
$\sigma\{T_u : s \le u \le t\}$-measurable for
$0 \le s < t < \infty$. However, there is no such representation
in which  $(T_t)_{t \ge 0}$ and $(U_t)_{t \ge 0}$  are independent.
\end{proposition}

\begin{proof}
For all $t \ge 0$ we have $x_t=Q_tT_t$, where
\begin{equation*}
Q_t=\frac{1}{\sqrt{(x^{11}_t)^2+(x^{21}_t)^2}}
\begin{pmatrix} x^{11}_t &-x^{21}_t \\ x^{21}_t&x^{11}_t\end{pmatrix}
\in K
\end{equation*}
and
\begin{equation*}
T_t=\begin{pmatrix}\sqrt{(x^{11}_t)^2+(x^{21}_t)^2}&\frac{x^{11}_t x^{12}_t +  x^{21}_t x^{22}_t}{\sqrt{(x^{11}_t)^2+(x^{21}_t)^2}}\\ 0& \frac{\det (x_t)}{\sqrt{(x^{11}_t)^2+(x^{21}_t)^2}}\end{pmatrix}
\in Y.
\end{equation*}

Note that $\det(x_t) = \det(T_t)$ and $\mathrm{tr}(x_t'x_t) = \mathrm{tr}(T_t'T_t)$, and so $f(x_t) = f(T_t)$.
Note also that the complex-valued process
$(x^{11}_t + i x^{21}_t)_{t \ge 0}$ is an isotropic complex
local martingale
in the sense of \cite[Ch 18]{K01}, that is
\[
[x^{11}]=[x^{21}]
\]
and
\[
[x^{22},x^{21}]=0.
\]
In our case

\[
d[x^{11}]_t = d[x^{21}]_t = f^2(T_t) \, dt.
\]
By \cite[Thm~18.5]{K01}, $(\log(x^{11}_t + i x^{21}_t)_{t \ge 0}$
is a well-defined isotropic complex local martingale that can
be written as
\[
\log(x^{11}_t + i x^{21}_t)  = \log \left(  T^{11}_t \right) + i \theta_t,
\]
where
\begin{equation*}
 d[\theta]_t = d[\log(T^{11})]_t= \frac{1}{(T^{11}_t)^2}d[x^{11}]_t  = \left( \frac{f(T_t)}{T^{11}_t} \right)^2 \, dt.
\end{equation*}

By the classical result of Dambis, Dubins and Schwarz
(see, for example, \cite[Thm~18.4]{K01}),
there exists a standard complex Brownian motion
$(\tilde{B}_t + iB_t)_{t \ge 0}$ such that $\log(x^{11}_t + i x^{21}_t) = \tilde{B}_{R_t} + iB_{R_t}$, where
\[
R_t = \int^t_0 \left( \frac{f(T_s)}{T_s^{11}} \right)^2 \, ds, \quad t \ge 0.
\]
So, $\theta_t = B_{R_t}$  and $\log(T^{11}_t) = \tilde{B}_{R_t}$.  Hence,
\[
\frac{x^{11}_t + i x^{21}_t}{\sqrt{(x^{11}_t)^2 + (x^{21}_t)^2}} =  \left( \cos(\theta_t) +  i \sin(\theta_t) \right)
\]
and
\begin{equation*}
Q_t  = \begin{pmatrix} \cos(B_{R_t} )& -\sin(B_{R_t})\\ \sin(B_{R_t})&  \cos(B_{R_t})\end{pmatrix}.
\end{equation*}
Consequently, $Q_t = U_{R_t}$, where
\begin{equation*}
U_t  = \begin{pmatrix} \cos(B_t )& -\sin(B_t)\\ \sin(B_t)&  \cos(B_t)\end{pmatrix},
\end{equation*}
and $(B_t)_{t \ge 0}$ is a standard one-dimensional Brownian motion.

Note that $(U_t)_{t \ge 0}$ is certainly a Brownian motion on $K=SO(2)$,
and so we have uniquely identified the $K$-valued Brownian motion
$(U_t)_{t \ge 0}$ and the increasing process $(R_t)_{t \ge 0}$
that appear in the claimed decomposition of $(x_t)_{t \ge 0}$.

To complete the proof, it suffices to suppose that
$(U_t)_{t \ge 0}$ is independent of $(T_t)_{t \ge 0}$
and obtain a contradiction.
An application of It\^o's Lemma shows that the entries of
$(U_t)_{t \ge 0}$ satisfy the system of SDEs
\begin{eqnarray*}\label{e_USDE}
dU^{1,1}_t &=& - U^{2,1}_t \, dB_t - \frac{1}{2}U^{1,1}_t \, dt\\
dU^{2,1}_t &=& U^{1,1}_t \, dB_t - \frac{1}{2}U^{2,1}_t \, dt\\
dU^{1,2}_t &=& -U^{1,1}_t \, dB_t + \frac{1}{2}U^{2,1}_t \, dt=-dU^{2,1}_t\\
dU^{2,2}_t &=& - U^{2,1}_t \, dB_t - \frac{1}{2}U^{1,1}_t \, dt =  dU^{1,1}_t.
\end{eqnarray*}

We apply Proposition \ref{p_timechange} below to each of the four SDEs in the system describing $(U_t)_{t \ge 0}$,  with, in the notation of that result, $(\zeta_t, H_t, K_t)$ being the respective triples
$(U^{1,1}_t ,U^{2,1}_t ,U^{1,1}_t)$,
$(U^{2,1}_t ,U^{1,1}_t ,U^{2,1}_t)$,
$(U^{1,2}_t ,U^{1,1}_t ,U^{2,1}_t)$,
and $(U^{2,2}_t ,U^{2,1}_t ,U^{1,1}_t)$.
In each of the four applications, we let
\begin{itemize}
\item $(\mathcal{F}_t)_{t \ge 0}$ be the filtration generated by $(U_t)_{t \ge 0}$,
\item $(\mathcal{G}_t)_{t \ge 0}$ be the filtration generated by $(T_t)_{t \ge 0}$,
\item $\beta_t = B_t$,
\item $\rho_t = R_t$,
\item $J_t =  \left( \frac{f(T_t)}{T_t^{11}} \right)^2$,
\item $\gamma_t = W_t = \int^t_0 \sqrt{\frac{1}{R^{\prime}_s}}dB_{R_s}$.
\end{itemize}

Let $\mathcal{H}_t = \mathcal{F}_{\rho_t} \vee \mathcal{G}_t$, $t \ge 0$,
as in the Proposition \ref{p_timechange}.
It follows by  the assumed independence of
$(U_t)_{t \ge 0}$ and $(T_t)_{t \ge 0}$,
part (iii) of Proposition \ref{p_timechange},
and equation \eqref{e_USDE} that the entries of the
time-changed process $Q_t = U_{R_t}$ satisfy the system of SDEs
\begin{eqnarray*}
dQ^{1,1}_t &=& - Q^{2,1}_t  \sqrt{R'_t} \, dW_t - \frac{1}{2}Q^{1,1}_t  R'_t  \, dt = - Q^{2,1}_t  \frac{f(T_t)}{T^{11}_t} \, dW_t - \frac{1}{2}Q^{1,1}_t  \left(\frac{f(T_t)}{T^{11}_t}\right)^2 \, dt\\
dQ^{2,1}_t &=& Q^{1,1}_t  \sqrt{R'_t} \, dW_t - \frac{1}{2}Q^{2,1}_t  R'_t \, dt = Q^{1,1}_t  \frac{f(T_t)}{T^{11}_t} \, dW_t - \frac{1}{2}Q^{2,1}_t  \left(\frac{f(T_t)}{T^{11}_t}\right)^2 \, dt\\
dQ^{1,2}_t &=& -dQ^{2,1}_t = Q^{1,1}_t  \sqrt{R'_t} \, dW_t - \frac{1}{2}Q^{2,1}_t  R'_t \, dt =  Q^{1,1}_t  \frac{f(T_t)}{T^{11}_t} \, dW_t - \frac{1}{2}Q^{2,1}_t  \left(\frac{f(T_t)}{T^{11}_t}\right)^2 \, dt\\
dQ^{2,2}_t &=& dQ^{1,1}_t = - Q^{2,1}_t  \sqrt{R'_t} \, dW_t - \frac{1}{2}Q^{1,1}_t  R'_t = - Q^{2,1}_t  \frac{f(T_t)}{T^{11}_t} \, dW_t - \frac{1}{2}Q^{1,1}_t \left(\frac{f(T_t)}{T^{11}_t}\right)^2 \, dt.
\end{eqnarray*}

Set
\begin{eqnarray*}
dw^1_t &=& \frac{x^{11}_t}{\sqrt{(x^{11}_t)^2+(x^{21}_t)^2}} \, dA^{11}_t +  \frac{x^{21}_t}{\sqrt{(x^{11}_t)^2+(x^{21}_t)^2}} \, dA^{21}_t\\
dw^2_t &=& \frac{-x^{21}_t}{\sqrt{(x^{11}_t)^2+(x^{21}_t)^2}} \, dA^{11}_t +  \frac{x^{11}_t}{\sqrt{(x^{11}_t)^2+(x^{21}_t)^2}} \, dA^{21}_t\\
dw^3_t &=& \frac{x^{11}_t}{\sqrt{(x^{11}_t)^2+(x^{21}_t)^2}} \, dA^{12}_t +  \frac{x^{21}_t}{\sqrt{(x^{11}_t)^2+(x^{21}_t)^2}} \, dA^{22}_t\\
dw^4_t &=& \frac{-x^{21}_t}{\sqrt{(x^{11}_t)^2+(x^{21}_t)^2}} \, dA^{12}_t +  \frac{x^{11}_t}{\sqrt{(x^{11}_t)^2+(x^{21}_t)^2}} \, dA^{22}_t.
\end{eqnarray*}
The processes $(w^i_t)_{t \ge 0}$ are local martingales with $[w^i_t,w^j_t]_t = \delta_{ij}t$, and thus they are independent standard Brownian motions.  An application of It\^o's Lemma shows that $(T_t)_{t \ge 0}$ is a diffusion satisfying the following system of SDEs.
\begin{eqnarray*}
dT^{11}_t &=& f(T_t) \, dw^1_t + \frac{f^2(T_t)}{T^{11}_t} \, dt\\
dT^{12}_t &=&  \frac{T^{22}_t f(T_t)}{T^{11}_t} \, dw^2_t + f(T_t)dw^3_t - \frac{T^{12}_t f^2(T_t)}{2 (T^{11}_t)^2} \, dt\\
dT^{22}_t &=& \frac{T^{12}_t f(T_t)}{T^{11}_t} \, dw^2_t + f(T_t)dw^4_t - \frac{T^{22}_t f^2(T_t)}{2 (T^{11}_t)^2} \, dt.
\end{eqnarray*}

The assumed independence of the processes
$(U_t)_{t \ge 0}$ and $(T_t)_{t \ge 0}$ and part (iv) of Proposition \ref{p_timechange} give that $[Q^{i,j},T^{k,l}] \equiv 0$ for all $i,j,k$ and $l$.
It follows from It\^o's Lemma that

\begin{eqnarray*}
d(Q_t T_t)^{1,1} &=& d N_t +\frac{Q_t^{1,1}f^2(T_t)}{T^{1,1}_t}\left(1-\frac{1}{2T_t^{1,1}}\right) \, dt,\\
\end{eqnarray*}
where $(N_t)_{t \ge 0}$ is a continuous local martingale
for the filtration $(\mathcal{H}_t)_{t \ge 0}$.  This, however,
is not possible because
$(Q_t T_t)^{1,1} = x_t^{1,1}$
and the process $(x^{1,1}_t)_{t \ge 0}$ is a continuous local martingale
for the filtration $(\mathcal{H}_t)_{t \ge 0}$.

\end{proof}

We required the following proposition that collects together some simple
facts about time-changes.

\begin{proposition}\label{p_timechange}
Consider two filtrations
$(\mathcal{F}_t)_{t \ge 0}$ and $(\mathcal{G}_t)_{t \ge 0}$
on an underlying probability space $(\Omega, \mathcal{F}, \mathbb{P})$.
Set $\mathcal{F}_\infty = \bigvee_{t \ge 0} \mathcal{F}_t$
and $\mathcal{G}_\infty = \bigvee_{t \ge 0} \mathcal{G}_t$.
Assume that the sub-$\sigma$-fields $\mathcal{F}_\infty$
and $\mathcal{G}_\infty$ are independent.  Suppose that
\[
\zeta_t = \zeta_0 + \int_0^t H_s \, d\beta_s + \int_0^t K_s \, ds,
\]
where $\zeta_0$ is $\mathcal{F}_0$-measurable, the integrands
$(H_t)_{t \ge 0}$
and $(K_t)_{t \ge 0}$ are $(\mathcal{F}_t)_{t \ge 0}$-adapted,
and $(\beta_t)_{t \ge 0}$ is an
$(\mathcal{F}_t)_{t \ge 0}$-Brownian motion.  Suppose further that $\rho_t = \int_0^t J_s \, ds$,
where $(J_t)_{t \ge 0}$ is a nonnegative, $(\mathcal{G}_t)_{t \ge 0}$-adapted process such that $\rho_t$
is finite for all $t \ge 0$ almost surely.
For $t \ge 0$ put
\[
\mathcal{F}_{\rho_t}
=
\sigma\{L_{s \wedge \rho_t} : \text{$s \ge 0$ and $L$ is
$(\mathcal{F}_t)_{t \ge 0}$-optional}\}.
\]
Set $\mathcal{H}_t = \mathcal{F}_{\rho_t} \vee \mathcal{G}_t$, $t \ge 0$.
Then the following hold.
\begin{itemize}
\item[(i)]
The process $(\beta_{\rho_t})_{t \ge 0}$ is a continuous local martingale
for the filtration $(\mathcal{H}_t)_{t \ge 0}$ with quadratic variation
$[\beta_{\rho_\cdot}]_t = \rho_t$.
\item[(ii)]
The process $(\gamma_t)_{t \ge 0}$, where
\[
\gamma_t = \int_0^t \sqrt{\frac{1}{J_s}} \, d \beta_{\rho_s},
\]
is a Brownian motion for the filtration $(\mathcal{H}_t)_{t \ge 0}$.
\item[(iii)]
If $\xi_t = \zeta_{\rho_t}$, $t \ge 0$, then
\[
\xi_t = \xi_0 + \int_0^t H_{\rho_s} \sqrt{J_s}\, d \gamma_s + \int_0^t K_{\rho_s} J_s \, ds.
\]
\item[(iv)]
If $(\eta_t)_{t \ge 0}$ is a continuous local martingale
for the filtration $(\mathcal{G}_t)_{t \ge 0}$, then
$(\eta_t)_{t \ge 0}$ is also a continuous local martingale
for the filtration $(\mathcal{H}_t)_{t \ge 0}$
and $[\eta,\gamma] \equiv 0$.
\end{itemize}
\end{proposition}

\begin{remark}\label{r:counter}
In this example $K/M = SO(2)$ has dimension $1$ and there is a type of skew-product decomposition. The angular part can indeed be written as a time-change depending on the radial part of a Brownian motion on $SO(2)$.
However, we cannot take this Brownian motion to be independent of the radial part.
\end{remark}

\section{Open problem}

The apparent counterexamples to \cite[Theorem~4]{L09} arise in Sections \ref{s:rotated} and \ref{s_counter} because $K/M$ is one-dimensional and hence trivially irreducible. When $K/M$ has dimension greater than $1$, irreducibility implies the nonexistence of a nonzero $M$-invariant tangent vector and it is this latter property that is actually used in the proof of  \cite[Theorem~4]{L09}.  In the examples in Sections \ref{s:BM}, \ref{s:rotated} and \ref{s_counter} the group $M$ is the trivial group consisting of just the identity and there certainly are nonzero $M$-invariant tangent vector.

Therefore, in view of the three examples we presented and Remarks \ref{r:BM}, \ref{r:rotated}, \ref{r:counter} we propose the following open problem.

\begin{question}
Suppose that $(x_t)_{t\geq 0}$ is a continuous Markov process with state space a smooth manifold $X$ and distribution that is equivariant under the smooth action of a Lie group $K$ on $X$ so that we can decompose $(x_t)_{t\geq 0}$ into a radial part $(y_t)_{t\geq 0}$ that is a Markov process on the submanifold $Y$ which is transversal to the orbits of $K$ and an angular part $(z_t)_{t\geq 0}$ that is a process on the homogeneous space $K/M$. Suppose further that $\text{dim}(K/M)=1$.
\begin{enumerate}
\item When can we write $z_t=B_{a_t}$ where $(B_t)_{t\geq 0}$ is a Brownian motion on $K/M$ and $(a_t)_{t\geq 0}$ is a time-change that is adapted to the filtration generated by $(y_t)_{t\geq 0}$.
\item Under which conditions can we take the Brownian motion $(B_t)_{t\geq 0}$ to be independent of $(x_t)_{t\geq 0}$?
\end{enumerate}
\end{question}

\subsection*{Acknowledgment}
We thank Prof. M. Liao for kindly explaining to us the role played by the assumption of irreducibility in \cite[Theorem~4]{L09}. We thank an anonymous referee for comments that improved this manuscript and for the example described in Section \ref{s:rotated}.

\bibliographystyle{amsalpha}
\bibliography{LIAO}

\providecommand{\bysame}{\leavevmode\hbox to3em{\hrulefill}\thinspace}
\providecommand{\MR}{\relax\ifhmode\unskip\space\fi MR }
\providecommand{\MRhref}[2]{%
  \href{http://www.ams.org/mathscinet-getitem?mr=#1}{#2}
}
\providecommand{\href}[2]{#2}
\begin{thebibliography}{ELJL10}

\bibitem[BN06]{Ba06}
F.~Baudoin and D.~Nualart, \emph{Notes on the two-dimensional fractional
  {B}rownian motion}, Ann. Probab \textbf{34} (2006), no.~1, 159--180.

\bibitem[Chy08]{Chy08}
O.~Chybiryakov, \emph{Skew-product representations of multidimensional {D}unkl
  {M}arkov processes}, Ann. Inst. Henri Poincar\'e Probab. Stat. \textbf{44}
  (2008), 593--611.

\bibitem[ELJL10]{El10}
K.~D. Elworthy, Y.~Le~Jan, and X.-M. Li, \emph{The geometry of filtering},
  Birkh\"auser Verlag, Basel, 2010.

\bibitem[EM91]{EM91}
A.~M. Etheridge and P.~March, \emph{A note on superprocesses}, Probab. Theory
  Related Fields \textbf{89} (1991), no.~2, 141--147.

\bibitem[Gal63]{G63}
A.~R. Galmarino, \emph{Representation of an isotropic diffusion as a skew
  product}, Probab. Theory Related Fields \textbf{1} (1963), no.~4, 359--378.

\bibitem[Hir00]{H00}
S.~Hiraba, \emph{Jump-type {F}leming-{V}iot processes}, Adv. in Appl. Probab.
  \textbf{32} (2000), no.~1, 140--158.

\bibitem[HJ13]{horn}
R.~A. Horn and C.~R. Johnson, \emph{Matrix analysis}, second ed., Cambridge
  University Press, Cambridge, 2013.

\bibitem[Kal02]{K01}
O.~Kallenberg, \emph{Foundations of modern probability}, Springer-Verlag, New
  York, 2002.

\bibitem[LCO09]{La09}
J.-A. L{\'a}zaro-Cam{\'{\i}} and J.-P. Ortega, \emph{Reduction, reconstruction,
  and skew-product decomposition of symmetric stochastic differential
  equations}, Stoch. Dyn. \textbf{9} (2009), no.~01, 1--46.

\bibitem[Lia09]{L09}
M.~Liao, \emph{A decomposition of {M}arkov processes via group actions}, J.
  Theoret. Probab. \textbf{22} (2009), no.~1, 164--185.

\bibitem[Per92]{P91}
E.~A. Perkins, \emph{Conditional {D}awson-{W}atanabe processes and
  {F}leming-{V}iot processes}, Seminar on {S}tochastic {P}rocesses, 1991 ({L}os
  {A}ngeles, {CA}, 1991), Progr. Probab., vol.~29, Birkh\"auser Boston, Boston,
  MA, 1992, pp.~143--156.

\bibitem[PR88]{PR88}
E.~J. Pauwels and L.~C.~G. Rogers, \emph{Skew-product decompositions of
  {B}rownian motions}, Contemp. Math. \textbf{73} (1988), 237--262.

\bibitem[RW00]{RW00}
L.~C.~G. Rogers and D.~Williams, \emph{Diffusions, {M}arkov processes, and
  martingales. {V}ol. 2: {I}t{\^o} calculus}, vol.~2, Cambridge University
  Press, Cambridge, 2000.

\bibitem[Tay92]{Tay92}
J.~C. Taylor, \emph{Skew products, regular conditional probabilities and
  stochastic differential equations: a technical remark}, S{\'e}minaire de
  Probabilit{\'e}s XXVI, Springer, Berlin, 1992, pp.~113--126.

\end{thebibliography}
\end{document}